\documentclass[reqno,12pt]{amsart}
\usepackage{amssymb}
\usepackage{amsmath, mathrsfs,bm}
\usepackage{enumerate,color,comment}
\usepackage{hyperref}

\makeatletter
\@namedef{subjclassname@2020}{%
  \textup{2020} Mathematics Subject Classification}
\makeatother

\newtheorem{thrm}{Theorem}[section]
\newtheorem{cor}[thrm]{Corollary}
\newtheorem{lem}[thrm]{Lemma}
\newtheorem{prop}[thrm]{Proposition}

\theoremstyle{definition}

\newtheorem{remark}[thrm]{Remark}

\numberwithin{equation}{section}

\newcommand{\CC}{{\mathbb C}}

\newcommand{\RR}{{\mathbb R}}

\newcommand{\rest}[2]{{{#1}_{\kern-.5pt|{#2}}}}

\newcommand{\id}{{\textup{id}}}  
\def\C*{{\sl C*}-algebra}
\def\Cs*{{\sl C*}-subalgebra}
\DeclareMathOperator{\rad}{rad}

\begin{document}

\title[Additive spectrum preserving mappings from von Neumann algebras]{Additive spectrum preserving mappings from\\ von Neumann algebras}

\author{Martin Mathieu}
\address{Mathematical Sciences Research Centre, Queen's University Belfast, Bel\-fast BT7 1NN, Northern Ireland}
\email{m.m@qub.ac.uk}
\author{Francois Schulz}
\address{Department of Mathematics and Applied Mathematics, Faculty of Science, University of Johannesburg,
P.O. Box 524, Auckland Park, 2006, South Africa}
\email{francoiss@uj.ac.za}


\date{\today}

\begin{abstract}
We establish Jafarian's 2009 conjecture that every additive spectrum preserving mapping from a von Neumann algebra
onto a semisimple Banach algebra is a Jordan isomorphism.
\end{abstract}

\maketitle

\section{Introduction}\label{sect:intro}
\noindent
In ~\cite{jafasurvey}, Jafarian asked whether a surjective unital additive mapping between two von Neumann algebras that preserves the spectrum of each element
is a Jordan isomorphism. It is the purpose of this note to answer his question in the affirmative.

For linear mappings, this was established by Aupetit in~\cite{Aup00}. Therefore, the main task is to prove that an additive mapping satisfying the above conditions
is indeed linear (by which we shall always mean \textit{complex-linear\/} in the following and emphasize \textit{real-linear\/} where necessary).
The study of the interplay between ring-theoretic and algebraic properties has a long history.
Already in 1954, Kaplansky showed that the domain of a ring isomorphism $\phi$ between two semisimple complex Banach algebras is a direct sum of a finite-dimensional
subalgebra and two subalgebras on which $\phi$ is either linear or conjugate-linear~\cite{kap1954}.
In the past few decades, the question when an additive mapping is automatically linear has been studied intensively in the context of preserver problems;
see, e.g., the surveys in~\cite{bourmash2015} and~\cite{bourmash2020}.
For example, the above-stated problem has been solved for the algebra of all bounded linear operators by Omladi\v c and \v Semrl in~\cite{omlsemrl1991}.

A breakthrough was obtained by Kowalski and S\l odkowski~\cite{KowalSlod} when they showed that every spectrally-additive scalar-valued mapping $\phi$
on a complex Banach algebra $A$ is a character. The assumption here is that $\phi(x)+\phi(y)\in\sigma(x+y)$ for all $x,y$ in~$A$,
where $\sigma(\,\cdot\,)$ denotes the spectrum. Their techniques were refined and extended by a number of authors subsequently, notably in~\cite{CostRep},
where a positive answer to Jafarian's question is obtained in the finite-dimensional case (under somewhat weaker hypotheses).
These techniques will be the basis for our approach too.
This consists, firstly, in carefully extending a variety of results which were previously known for linear spectrum preserving mappings to the additive setting,
largely for general semisimple Banach algebras. Using the by-now well-known fact that a continuous, unital, linear mapping which preserves idempotents is a
Jordan homomorphism, provided the domain is rich in idempotents, we turn our attention to \C*s of real rank zero.
In a final step, the techniques hinted at above together with some classical results on (Jordan) derivations and automorphisms will provide the
desired answer (Theorem~\ref{youbeauty}).

In the original formulation of his conjecture \cite[Conjecture~2]{jafasurvey}, Jafarian asked for a characterisation of \textit{invertibility
preserving\/} additive mappings as Jordan isomorphisms. This, however, is not possible as an example in \cite{omlsemrl1991} shows;
such mappings need not be continuous.

\section{Some Properties of Additive Spectral Preservers}\label{sect:prop}
\noindent
Throughout we assume that $A$ and $B$ are at least complex unital Banach algebras.
We denote by $Z(A)$ the centre of an algebra $A$; by $\rad(A)$ its Jacobson radical; and by $G(A)$ its group of invertible elements.
For an element $x\in A$, $r(x)$ stands for its spectral radius.
In this section, we collect various observations which are well known for linear mappings but turn out to be valid
in a much wider setting.

\begin{prop}\label{propcent}
Suppose that $A$ is semisimple and that $T\colon A\to B$ is a surjective map with the property that $r(x+y) = r(Tx+Ty)$ for all $x, y\in A$.
Then $B$ is semisimple and $Z(B) = T(Z(A))$.
\end{prop}
\begin{proof}
Firstly, notice that $r(x+x) = r(Tx+Tx)$ for all $x \in A$ implies that $T$ preserves the spectral radius, i.e., $r(x) = r(Tx)$ for all $x \in A$.
In particular, this means that $q$ is quasinilpotent in $A$ if and only if $Tq$ is quasinilpotent in~$B$.
Since $T$ is surjective, every element in $\rad(B)$ has the form $Ta$ for some $a \in A$
and it follows that
	$$r\left(Ta + Tq\right) = 0 \mbox{ for each quasinilpotent }q \in A.$$
Thus, by hypothesis, we obtain
	$$r(a+q) = 0 \mbox{ for each quasinilpotent }q \in A.$$
Since $A$ is semisimple, an application of Zem\'anek's characterization of the radical  \cite[Theorem 5.3.1]{aupetit1991primer}
yields that $a = 0$. Hence, $Ta = T(0)$.
This shows that $\{0\}\subseteq\rad(B)\subseteq\{T(0)\}$ which entails that $T(0)=0$ and that $B$ is semisimple.

For the second part, assume that $x \in Z(A)$. By \cite[Theorem 5.2.2]{aupetit1991primer}, this is characterized by the fact that there exists some $M > 0$ such that
	$$r(x+y) \leq M(1+r(y)) \mbox{ for all }y \in A.$$
	Consequently, by hypothesis, we have
	$$r(Tx + Ty) \leq M(1+r(Ty)) \mbox{ for all }y \in A.$$
	Because $T$ is surjective, we can now use \cite[Theorem 5.2.2]{aupetit1991primer} once more to conclude that $Tx \in Z(B)$.
In a similar way one infers that $x \in Z(A)$ whenever $Tx \in Z(B)$. Hence, $Z(B) = T\left(Z(A)\right)$ as desired.
\end{proof}
The following result is the starting point in~\cite{Askes2022}. We point out that statement~(b) depends on \cite[Theorem 2.2]{BraatBrit}.
\begin{prop}[{\cite[Proposition 1.3]{Askes2022}}]\label{propbasic}
Suppose that $A$ is semisimple and that $T\colon A\to B$ is a surjective map which is spectrally additive, that is, $\sigma(x+y) = \sigma(Tx+Ty)$ for each $x, y\in A$. Then
	\begin{itemize}
		\item[\textnormal{(a)}]
		$\sigma(x) = \sigma(Tx)$ for each $x \in A$. Consequently, $T(G(A)) = G(B)$.
		\item[\textnormal{(b)}]
		$T$ is injective.
		\item[\textnormal{(c)}]
		$T \left(\lambda \mathbf{1} + x\right) = \lambda \mathbf{1} + T(x)$ for each $\lambda \in \CC$ and $x \in A$.
        In particular, $T(0)=0$, $T$ is unital and homogeneous at~$\mathbf{1}$.
	\end{itemize}
\end{prop}

Combining Proposition~\ref{propcent} and part (c) in Proposition~\ref{propbasic}
with \cite[Theorem 2.1]{Hatori2010} we now obtain the following corollaries.
\begin{cor}
If $T\colon A\to B$ is a surjective spectrally additive map, and $A$ is commutative and semisimple, then $T$ is an algebra isomorphism.
\end{cor}
\begin{cor}\label{corcent}
If $A$ is semisimple and $T\colon A\to B$ is a surjective map which is unital, homogeneous at $\mathbf{1}$,
and has the property that $r(x+y) = r(Tx+Ty)$ for all $x, y \in A$,
then the restriction $T_|\colon Z(A)\to Z(B)$ is an algebra isomorphism.
In particular, the above hypothesis on $T$ is satisfied if\/ $T\colon A \to B$ is a surjective spectrally additive map.
\end{cor}

The next result leads to an extension of Aupetit's lemma on the automatic continuity of surjective spectrally bounded operators
(see \cite[Theorem 5.5.1, Theorem 5.5.2]{aupetit1991primer}). By modifying Aupetit's proof slightly, it is possible to show that any
surjective additive spectrum preserving map on a semisimple Banach algebra is continuous.
In fact, it is enough to know that the additive map in question preserves the spectral radius. The details are as follows.

\begin{lem}\label{autocontlem}
Suppose that $T\colon A\to B$ is a surjective additive map that preserves the spectral radius. Then
$$b\in\mathscr{S}(T):= \left\{y\in B:\lim_{n\rightarrow\infty} x_n = 0\mbox{ and }\lim_{n\rightarrow\infty} Tx_n = y\mbox{ for some }(x_n)\subseteq A\right\}$$
implies $r(y)\leq r(b+y)$ for all $y\in B$. In particular, $\mathscr{S}(T)$ is contained in the set of quasinilpotent elements of~$B$.
\end{lem}
\begin{proof}
Let $b \in \mathscr{S}(T)$ be given, and assume that $(x_n)\subseteq A$ satisfies
$$\lim_{n \rightarrow \infty} x_n = 0 \mbox{ and } \lim_{n \rightarrow \infty} Tx_n = b.$$
Let $x \in A$ and $\lambda \in \mathbb{C}\setminus\{0\}$ be arbitrary. Since $T$ is surjective, there exists some $a \in A$ (depending on $\lambda$)
such that $Tx = \lambda Ta$. In particular, $r(Tx) = |\lambda| r(Ta)$. Thus, since $T$ preserves the spectral radius,
it follows that $r(x) = |\lambda|r(a) = r(\lambda a)$.
From $\lim\limits_{n \rightarrow \infty} \lambda x_n + \lambda a = \lambda a$ and
$$r\left(\lambda Tx_n + Tx \right)  =  r\left(\lambda Tx_n + \lambda Ta \right)
 =  |\lambda|r\left(Tx_n + Ta \right)
 =  r\left(\lambda x_n + \lambda a \right)$$
we obtain
$$\limsup_{n \rightarrow \infty} r\left(\lambda Tx_n + Tx \right) = \limsup_{n \rightarrow \infty} r\left(\lambda x_n + \lambda a \right) \leq r(\lambda a) = r(x),$$
by the upper semicontinuity of the spectral radius on~$A$.
Since $\limsup\limits_{n \rightarrow \infty} r\left(\lambda Tx_n + Tx \right) = r(Tx) = r(x)$ if $\lambda = 0$, we conclude that
$$\limsup_{n \rightarrow \infty} r\left(\lambda Tx_n + Tx \right) \leq r(x) \mbox{ for all } \lambda \in \mathbb{C}.$$
We now proceed exactly as in Aupetit's proof of \cite[Theorem 5.5.1]{aupetit1991primer}.
For each $n \in \mathbb{N}$, define $\phi_n(\lambda) = r\left(\lambda Tx_n + Tx \right)$ for each $\lambda \in \mathbb{C}$
which is subharmonic by Vesentini's theorem \cite[Theorem 3.4.7]{aupetit1991primer}. In particular, we have that
$$\phi(\lambda) = \limsup_{n \rightarrow \infty} \phi_n(\lambda) \leq r(x)$$
satisfies the mean inequality on $\mathbb{C}$, but in general it is not upper semicontinuous. We therefore consider its upper regularization,
$$\psi(\lambda) = \limsup_{\mu \rightarrow \lambda} \phi(\mu),$$
which is subharmonic on $\mathbb{C}$. Observe now that $\phi(\lambda) \leq \psi(\lambda)\leq r(x)$ for all $\lambda\in\mathbb{C}$.
Hence, by Liouville's theorem for subharmonic functions \cite[Theorem A.1.11]{aupetit1991primer}, we infer that $\psi$ is constant on~$\mathbb{C}$. So,
$$r(Tx) = \phi(0) \leq \psi(0) = \psi(\lambda) \mbox{ for all } \lambda \in \mathbb{C}.$$
Using the upper semicontinuity of $r$ on $B$ we now infer that
$$\phi(\lambda) \leq r(\lambda b + Tx) \mbox{ for all } \lambda \in \mathbb{C},$$
and consequently that
$$\psi(\lambda) \leq \limsup_{\mu \rightarrow \lambda} r(\mu b + Tx) \leq r(\lambda b + Tx) \mbox{ for all } \lambda \in \mathbb{C}.$$
Together we obtain that $r(Tx) \leq r(\lambda b + Tx)$ for all $\lambda \in \mathbb{C}$. This holds in particular for $\lambda = 1$.
For the last part of the lemma, let $b \in \mathscr{S}(T)$ with $b = Tu$ for some $u \in A$. Taking $x = -u$ and using the additivity of $T$, we get $r(b) = 0$ as desired.
\end{proof}

\begin{thrm}\label{autocontthm}
Let $A$ be semisimple. Suppose that $T\colon A\to B$ is a surjective additive map that preserves the spectral radius. Then $T$ is continuous.
\end{thrm}
\begin{proof}
By hypothesis and Proposition~\ref{propcent}, it follows that $B$ is semisimple. Let $b \in \mathscr{S}(T)$ with $b = Tu$.
Taking $x = a-u$, we get $r(Ta-b) \leq r(Ta)$ for all $a \in A$. Consequently, $r(b+q) = 0$ for all quasinilpotent elements $q$ of $B$.
By Zem\'anek's characterization of the radical, we conclude that $b \in \rad(B) = \{0\}$.
This shows that $\mathscr{S}(T) = \{0\}$. Since $T$ is additive, and thus $T(-a)=-Ta$ for all~$a$, it follows that $Tx_n \rightarrow Tx$ in $B$
whenever $x_n \rightarrow x$ in $A$. Hence, $T$ is continuous.
\end{proof}
\begin{remark}\label{rem:just-semisimple}
It is easy to see that Theorem~\ref{autocontthm} remains valid under the (potentially) weaker assumption that $B$ is semisimple,
just as in Aupetit's original result. On the other hand, it is not sufficient to assume that the additive mapping is merely invertibility preserving,
as the example in~\cite{omlsemrl1991} mentioned at the end of our Introduction shows.
\end{remark}
In~\cite{parv2017}, an erroneous proof of Theorem~\ref{autocontthm} under the stronger assumption of spectrum preserving appeared.
However, the attempted adaptation of Ransford's beautiful proof of Aupetit's lemma (see \cite[proof of 6.4.4]{ransford1995}) fails and therefore the main result of that paper
(a partial positive answer to Jafarian's question) fails as well.

\smallskip
Every continuous additive mapping between complex Banach algebras is linear over~$\RR$.
Consequently, such a map is readily seen to be bounded. Assume now that $T\colon A \to B$ is a surjective additive spectrum preserving map with $A$ semisimple.
Then, by Proposition~\ref{propbasic}, $B$ is semisimple and $T^{-1}$ exists and is an additive spectrum preserving map.
Hence, by the preceding discussion, we conclude that there exist $\alpha > 0$ and $\beta > 0$ such that
\begin{equation}\label{eq:bdd-below}
\alpha \|x\| \leq \|Tx\| \leq \beta \|x\| \mbox{ for all } x \in A.
\end{equation}
We are now in a position to take the following important step forward.
\begin{prop}\label{idempotentthm}
Let $A$ be semisimple. Suppose that $T\colon A \to B$ is a surjective additive spectrum preserving map. Then $T$ maps idempotents in $A$ to idempotents in~$B$.
\end{prop}
\begin{proof}
Let an idempotent $e$ in $A$ be given. By the characterization given in \cite[Corollary 1.3]{ransford2001},
there are real numbers $C > 0$ and $\delta > 0$ such that
$\Delta(\sigma(e+x), \sigma(e))\leq C\|x\|$ for all $x\in A$ with $\|x\|<\delta$, where $\Delta$ denotes the Hausdorff distance.
Since $T$ is additive and spectrum preserving, it follows that $\sigma(Te) \subseteq \{0, 1\}$ and
$$\Delta(\sigma(e+x), \sigma(e)) = \Delta(\sigma(Te+Tx), \sigma(Te)).$$
Thus, using the notation in~\eqref{eq:bdd-below},
$$\Delta(\sigma(Te+Tx), \sigma(Te))\leq\frac{C}{\alpha} \|Tx\|$$
for all $Tx \in B$ with $\|Tx\| <\alpha\delta$. Since $T$ is surjective, we can use \cite[Corollary 1.3]{ransford2001} once more to conclude that $Te$ is an idempotent in~$B$.
\end{proof}
An argument similar to the proof of Proposition~\ref{idempotentthm}, which uses \cite[Corollary 1.4]{ransford2001}, yields the following result;
see also \cite[Lemma~3.1]{MathieuSchick2003}.
\begin{prop}\label{sqzerothm}
Let $A$ be semisimple. Suppose that $T\colon A\to B$ is a surjective additive spectrum preserving map.
Then $T$ maps square-zero elements in $A$ to square-zero elements in~$B$.
\end{prop}

\section{Additive Spectrum Preserving Mappings On Certain \textsl{C*}-Algebras}\label{sect:addspecCstar}
\noindent
In order to put the results in the previous section to good use, we need to make further assumptions on the domain algebra.
A unital \C* $A$ is of \textit{real rank zero\/} if the invertible self-adjoint elements are dense in the set $A_{sa}$
of self-adjoint elements of~$A$ \cite[V.3.2.7]{blackadar}.
This is equivalent to the property that every $a\in A_{sa}$ is the limit of finite (real-)linear combinations of orthogonal projections
in~$A$. This property together with the fact that spectrum preserving linear mappings (and in fact, even spectrally bounded linear mappings,
see, e.g., \cite{MathieuSchick2003}) send orthogonal projections to orthogonal idempotents has been the key to many results
leading to the conclusion of a Jordan isomorphism. A prominent example is Aupetit's paper~\cite{Aup00}.

An alternative, independent approach puts constraints on the codomain algebra, see, e.g.,~\cite{MathieuSchulz2023}.

The continuity of an additive spectrum preserving map together with the preservation of idempotents (Proposition~\ref{idempotentthm}) leads to the following result.

\begin{prop}\label{weakjordanprop}
Let $A$ be a \C* with real rank zero and suppose that $T\colon A\to B$ is a surjective additive spectrum preserving map.
Then $T(a^2) = (Ta)^2$ for every $a\in A_{sa}$.
\end{prop}
\begin{proof}
By Theorem~\ref{autocontthm}, $T$ is an $\RR$-linear continuous mapping. By Proposition~\ref{idempotentthm}, $T$ preserves idempotents and thus,
$T$ preserves orthogonality between idempotents. If $x$ is a real-linear combination of mutually orthogonal projections,
then $T(x^2) = (Tx)^2$, as is readily checked. Since $A$ has real rank zero and $T$ is continuous, this establishes that $T(a^2) = (Ta)^2$
for every self-adjoint element $a\in A$.
\end{proof}

The next result allows us to extend the restriction of $T$ on the self-adjoint part of $A$ to a Jordan homomorphism on all of~$A$.

\begin{prop}\label{genjordanprop}
Let $A$ be a \C* with real rank zero and suppose that $T\colon A\to B$ is a surjective additive spectrum preserving map.
Then there exists a unital Jordan homomorphism $S\colon A\to B$ which agrees with\/ $T$ on~$A_{sa}$.
\end{prop}
\begin{proof}
Since each $x\in A$ has a unique representation $x =a + ib$, where $a$ and $b$ are self-adjoint elements in $A$, the map $S\colon A\to B$ given by $Sx=Ta+iTb$ is well defined.

Let $x = a +ib$ and $y = h+ik$ with $a, b, h, k \in A_{sa}$. Then, since $a+h$ and $b+k$ are self-adjoint and $T$ is additive, we get
\begin{equation*}
\begin{split}
	S(x+y) & = S((a+h) + i(b+k))  =  T(a+h) + iT(b+k) \\
	 & =  Ta + iTb + Th +iTk  =  Sx + Sy.
\end{split}
\end{equation*}
Hence, $S$ is additive. Moreover, since any real scalar multiple of a self-adjoint element is self-adjoint, and since $T$ is linear over~$\mathbb{R}$,
for any $\alpha \in \mathbb{C}$ we have

\begin{eqnarray*}
	S(\alpha x) & = & S((\mathrm{Re}(\alpha)+i\mathrm{Im}(\alpha))x) \\
	& = & S(\mathrm{Re}(\alpha)a+i\mathrm{Re}(\alpha)b -\mathrm{Im}(\alpha)b + i\mathrm{Im}(\alpha)a) \\
	& = & S(\mathrm{Re}(\alpha)a+i\mathrm{Re}(\alpha)b) +S(-\mathrm{Im}(\alpha)b + i\mathrm{Im}(\alpha)a) \\
	& = & T(\mathrm{Re}(\alpha)a)+iT(\mathrm{Re}(\alpha)b) +T(-\mathrm{Im}(\alpha)b) + iT(\mathrm{Im}(\alpha)a) \\
	& = & \mathrm{Re}(\alpha)Ta+i\mathrm{Re}(\alpha)Tb -\mathrm{Im}(\alpha)Tb + i\mathrm{Im}(\alpha)Ta \\
	& = & (\mathrm{Re}(\alpha)+i\mathrm{Im}(\alpha))Ta + i(\mathrm{Re}(\alpha)+i\mathrm{Im}(\alpha))Tb \\
	& = & \alpha(Ta + iTb) \;\, = \;\, \alpha Sx.
\end{eqnarray*}
Hence, $S$ is (complex-)linear. It remains to show that $S(x^2) = (Sx)^2$ for all $x \in A$.
From Proposition~\ref{weakjordanprop}, if $a$ and $b$ are self-adjoint in $A$, then $T((a+b)^2) = (T(a+b))^2$ from which we infer that
$T(ab+ba) = TaTb + TbTa$. Let $x = a +ib$ for $a$ and $b$ in~$A_{sa}$.
Notice that $a^2-b^2$ and $ab+ba$ are both self-adjoint in $A$. Hence,
\begin{equation*}
\begin{split}
		S(x^2) & = S(a^2-b^2+i(ab+ba)) 	= T(a^2-b^2)+iT(ab+ba) \\
		& = (Ta)^2-(Tb)^2+i(TaTb+TbTa)  = (Ta+iTb)^2\\
		& = (Sx)^2.
\end{split}
\end{equation*}
This completes the proof.
\end{proof}
At this stage it is natural to ask whether the mapping $S$ is surjective or injective. The second question is easily settled.

Since $S$ is a unital Jordan homomorphism, $\sigma(Sx)\subseteq\sigma(x)$ for every $x\in A$.
However, as $S$ agrees with $T$ on $A_{sa}$ and $T$ is spectrum preserving, $\sigma(Sx)=\sigma(x)$ for every $x\in A_{sa}$.
Let $x=a+ib$ with $a,b\in A_{sa}$ and suppose that $Sx=0$. Then $\sigma(a)=\sigma(Sa)=\sigma(-S(ib))=-i\sigma(b)$
which is only possible if $a=b=0$; thus $S$ is injective.

As a direct argument for the surjectivity of $S$ eludes us (which, in fact, would immediately imply the linearity of~$T$),
we resort to a trick which is reminiscent of the methods in~\cite{CostRep} and~\cite{mrcun}.

For the remainder of this section, we shall assume that $A$ is a unital \C* with real rank zero which is the domain
of our spectrum preserving additive mapping $T$ onto the unital Banach algebra~$B$. Note that, by Proposition~\ref{propcent},
$B$ is semisimple.

Set $R=T^{-1}S$. By the above discussion, $R$ is a unital, $\RR$-linear, injective continuous mapping which preserves the spectrum
of self-adjoint elements. Our aim is to show that $R=\id$ so that $T=S$ is linear. (We already know that $Rx=x$ for all $x\in A_{sa}$.)
\begin{thrm}\label{genmapthm}
The mapping $\tilde{R}\colon A\to A$, defined by
$$\tilde{R}x = \textstyle{\frac{1}{2}}\left(Rx-iR(ix)\right) \mbox{ for each }x\in A,$$
is a unital Jordan homomorphism on~$A$.
\end{thrm}
\begin{proof}
By Proposition~\ref{propbasic} (c), $R(i\mathbf{1})=i\mathbf{1}$ hence $\tilde R$ is unital.
Since $R$ is continuous, additive and $\RR$-linear, the same is true for~$\tilde{R}$.
Thus, to establish the linearity of $\tilde{R}$, it is enough to show that $\tilde{R}(ix) = i\tilde{R}(x)$ for each $x\in A$,
which is routine to check.

We now show that $\tilde{R}$ preserves idempotents. To this end, let $p\in A$ be an idempotent, and set $e\!:= Rp$ and $f\!:=-iR(ip)$.
By Proposition~\ref{idempotentthm}, $T$ and hence $T^{-1}$ preserve idempotents and so does $S$ as a Jordan homomorphism.
It follows that $e$ and $f$ are both idempotents. Moreover, using similar reasoning, it follows that the function
	$$x \mapsto \frac{(1-i)}{2}R((1+i)x) \qquad(x \in A)$$
	also preserves idempotents; hence, $\frac{(1-i)}{2}R((1+i)p)$ is an idempotent as well. Notice now that
	\begin{equation}
	\frac{(1-i)}{2}R((1+i)p) = \frac{1}{2}((1-i)e+(i+1)f),
	\label{eq4}
	\end{equation}
	and that
	\begin{eqnarray}
		\left(\frac{1}{2}((1-i)e+(i+1)f)\right)^2 & = & \frac{1}{4}\left((1-i)^2e + (i+1)^2f + 2ef + 2fe \right) \nonumber \\
		& = & \frac{1}{4}\left(-2ie + 2if + 2ef + 2fe\right) \nonumber \\
		& = & \frac{1}{2}\left(-ie + if + ef + fe\right). \label{eq5}
	\end{eqnarray}
	Since $\frac{(1-i)}{2}R((1+i)p)$ is an idempotent, it follows from \eqref{eq4} and \eqref{eq5} that
	$$e + f = ef + fe.$$
	Consequently, $e-f$ is a square zero element and for any $\lambda \in \mathbb{C}$, we have that
	\begin{eqnarray*}
		(\lambda e + (1-\lambda)f)^2 & = & \lambda^2 e + (1-\lambda)^2 f +\lambda(1-\lambda)(ef+fe) \\
		& = & \lambda^2(e+f) +(1-2\lambda)f + \lambda(ef+fe) - \lambda^2(ef+fe) \\
		& = & \lambda^2(e+f) +(1-2\lambda)f + \lambda(e+f) - \lambda^2(e+f) \\
		& = & \lambda e + (1-\lambda)f.
	\end{eqnarray*}
	Hence, $\lambda e + (1-\lambda)f$ is an idempotent for every $\lambda \in \mathbb{C}$. With $\lambda = \frac{1}{2}$, we see that
	$$\tilde{R}p = \frac{1}{2}\left(Rx-iR(ix)\right) = \frac{1}{2}(e+f)$$
	is an idempotent.

It is well known that any continuous, linear, and idempotent preserving map on a $C^{\ast}$-algebra with real rank zero is a Jordan homomorphism;
compare Proposition~\ref{weakjordanprop} and~\cite{MathieuSchick2003}.
\end{proof}
\begin{remark}\label{rem:all-idempotents}
Note that the argument in the above proof in fact shows that, for every $\lambda\in\CC$ and each idempotent $p\in A$,
\begin{equation}\label{eq1}
\lambda Rp + (1-\lambda)(-iR(ip))
\end{equation}
is an idempotent in~$A$.
\end{remark}

\begin{lem}\label{lem3}
For each $x\in A$, $\left(Rx + iR(ix)\right)^2 = 0$.
\end{lem}
\begin{proof}
Firstly, let $x\in A_{sa}$. Then $(Rx)^2=x^2=R(x^2)$.
By Remark~\ref{rem:all-idempotents}, we have that
	\begin{equation}\label{eq2}
	(-iR(ip))^2 = -iR(ip)
	\end{equation}
for each idempotent $p\in A$.
Since $R$ is additive, this implies that the mapping $y\mapsto -iR(iy)$ sends orthogonal idempotents to orthogonal idempotents.
Since $A$ has real rank zero and $R$ is continuous and $\RR$-linear, we conclude that
$$-R(ix)^2 = (-iR(ix))^2 = -iR(ix^2).$$
By Theorem~\ref{genmapthm}, $\tilde{R}$ is a Jordan homomorphism. Consequently,
	\begin{eqnarray*}
		\frac{1}{2}\left(R(x^2)-iR(ix^2)\right) & = & \left(\frac{1}{2}\left(Rx-iR(ix)\right) \right)^2 \\
		& = & \frac{1}{4}\left((Rx)^2-(R(ix))^2 - i\left(RxR(ix) + R(ix)Rx\right)\right) \\
		& = & \frac{1}{4}\left(R(x^2)-iR(ix^2) - i\left(RxR(ix) + R(ix)Rx\right)\right),
	\end{eqnarray*}
	and so,
	$$R(x^2)-iR(ix^2) = -i\left(RxR(ix) + R(ix)Rx\right).$$
It follows that
	$$\left(Rx + iR(ix)\right)^2  =  (Rx)^2-(R(ix))^2 + i\left(RxR(ix) + R(ix)Rx\right) = 0$$
	for all $x \in A_{sa}$.

Now let $x, y \in A_{sa}$ be arbitrary. Since
	$$\left(Rx + iR(ix)\right)^2 = \left(Ry + iR(iy)\right)^2  =\left(R(x+y) + iR(i(x+y))\right)^2 = 0,$$
	the additivity of $R$ yields that
	$$\left(Rx + iR(ix)\right)\left(Ry + iR(iy)\right) + \left(Ry + iR(iy)\right)\left(Rx + iR(ix)\right) = 0.$$
	Hence,
	\begin{eqnarray*}
		& & \left(R(x+iy) + iR(i(x+iy))\right)^2 \\ & = & \left((Rx+iR(ix))-i(Ry+iR(iy))\right)^2 \\
		& = & (Rx+iR(ix))^2 - (Ry+iR(iy))^2 \\ & &  {}\quad-i \left(\left(Rx + iR(ix)\right)\left(Ry + iR(iy)\right) + \left(Ry + iR(iy)\right)\left(Rx + iR(ix)\right)\right) \\
		& = & 0.\qedhere
	\end{eqnarray*}
\end{proof}

The bounded linear mapping $W\colon A\to A$, defined by
$$Wx = \frac{1}{2}\left(R(x^{\ast}) +iR(ix^{\ast})\right)\;\;\;(x \in A),$$
maps into the set of square zero elements, by Lemma~\ref{lem3}. Our next result highlights the interaction between $\tilde{R}$ and~$W$.
\begin{lem}\label{lem4}
For every $\alpha\in\CC$, $\tilde{R} +\alpha W$ is a (continuous) Jordan homomorphism from $A$ into itself.
Moreover, $\tilde{R}+W =\id$. Thus, for all $\alpha$ sufficiently close to~$1$, it follows that $\tilde{R} +\alpha W$ is a Jordan automorphism.
\end{lem}
\begin{proof}
To see that $\tilde{R}+W =\id$, simply note that, for every $x\in A$, we have
	\begin{eqnarray*}
		\tilde{R}x + Wx & = & \frac{1}{2}\left(Rx -iR(ix)\right) + \frac{1}{2}\left(R(x^{\ast}) +iR(ix^{\ast})\right) \\
		& = & R\left(\frac{x+x^{\ast}}{2} \right) + iR\left(\frac{x-x^{\ast}}{2i} \right) \\
		& = & \frac{x+x^{\ast}}{2} + i\frac{x-x^{\ast}}{2i} \\
		& = & x.
	\end{eqnarray*}
Fix $\alpha\in\CC$. In order to prove that $\tilde{R} +\alpha W$ is a Jordan homomorphism,
it suffices to show that $\tilde{R} +\alpha W$ maps projections to idempotents. To this end, notice that for any projection $p = p^*\in A$, we have
	\begin{eqnarray*}
		\tilde{R}p + \alpha Wp & = & \frac{1}{2}\left(Rp -iR(ip)\right) + \frac{\alpha}{2}\left(Rp +iR(ip)\right) \\
		& = & \left(\frac{1}{2}+\frac{\alpha}{2}\right) Rp + \left(\frac{1}{2}-\frac{\alpha}{2}\right)\left(-iR(ip)\right) \\
		& = & \lambda Rp +(1-\lambda)\left(-iR(ip)\right),
	\end{eqnarray*}
with $\lambda = (1+\alpha)/2$, and this is an idempotent by Remark~\ref{rem:all-idempotents}.

Since the group of invertible bounded linear mappings on $A$ is open, it follows that $\tilde{R} +\alpha W$ is bijective for $\alpha$ in a neighbourhood
of~$1$.
\end{proof}
\begin{remark}\label{rem:decompose}
The strategy to decompose the identity on~$A$ into $\tilde R$ and $W$ is motivated by the methods in~\cite{KowalSlod} and~\cite{CostRep}.
In both papers, the original additive surjective spectrum preserving mapping $T\colon A\to B$ gives rise to the mapping $h(t)=e^{it}T(e^{-it}x)$,
from $\RR$ into~$B$, for fixed~$x$, which is then decomposed as follows
\begin{equation}\label{eq:decompose}
h(t)=\textstyle{\frac12}(Tx-iT(ix))+\textstyle{\frac12}e^{2it}(Tx+iT(ix)),
\end{equation}
since $T$ is $\RR$-linear. In~\cite{KowalSlod}, a secondary decomposition (which is only available in the case $B=\RR$) together with the spectrum preserving
property leads to the conclusion that $Tx=-iT(ix)$, for every~$x$, so that $T$ is linear \cite[Lemma~2.1]{KowalSlod}.

In~\cite{CostRep}, $A=B=M_n(\CC)$, therefore injectivity already implies bijectivity and the well-known characterisations of linear spectrum preserving
maps guide the authors through the various steps from~\eqref{eq:decompose} to the Jordan automorphism $Kx=\frac12(x+\tilde Rx)$ (in our notation),
which is then shown to be the identity (\cite[Lemma~2.1 and Lemma~1.3]{CostRep}), that is, $\tilde Rx=x$ for all $x\in A$.

In both approaches the conclusion from
\[
\sigma(x)=\sigma(Tx)=\sigma\bigl(\textstyle{\frac12}(Tx-iT(ix))+\textstyle{\frac{\lambda}{2}}(Tx+iT(ix))\bigr)
\]
for all $\lambda\in\CC$ with $|\lambda|=1$ to the same identity for \textit{all\/} $\lambda\in\CC$ is decisive.
Using \cite[Chapter~3, Exercise~19]{aupetit1991primer}, one can show that this identity remains valid for all $\lambda$ with $|\lambda|\leq1$ and all $x\in A_{sa}$.
\end{remark}
Although we do not need the next result in what follows,
we record it here as another instance of how similar additive mappings can behave in comparison with linear mappings.
It is inspired by the work in \cite{MathieuSourour2004} on spectral isometries;
the proof uses an idea from the argument which ultimately establishes \cite[Corollary 8]{MathieuSourour2004}.
\begin{prop}\label{weakcentralprojprop}
Let $A$ be a unital \C* with real rank zero and suppose that $T\colon A\to B$ is a surjective additive spectrum preserving map.
If $e$ is a central projection in~$A$, then $T(ex) = TeTx$ for all $x\in A$.
\end{prop}
\begin{proof}
Fix a central projection $e\in A$ and recall from Corollary~\ref{corcent} that the restriction of $T$ to $Z(A)$ is an algebra isomorphism onto~$Z(B)$.
Therefore, $Te$ is a central idempotent in $B$ and $T$ is homogeneous at~$e$. Let $y\in A_{sa}$.
By Proposition~\ref{weakjordanprop}, it follows that $T(y^2) = (Ty)^2$ and
	\begin{equation*}
		T((e+y)^2) = (T(e+y))^2.
	\end{equation*}
Hence, since $e$ and $Te$ are central, $T(ey) = TeTy$ for each self-adjoint $y\in A$.
Let $x = a + ib$, where both $a$ and $b$ are self-adjoint.
Since the map $x \mapsto -iT(ix)$ is a surjective additive spectrum preserving map from $A$ onto~$B$,
	\begin{eqnarray*}
		-i(\mathbf{1}-Te)T(ex) & = & -i(\mathbf{1}-Te)T(e(a+ib)) \\
		& = & -i(\mathbf{1}-Te)T(ea) -i(\mathbf{1}-Te)T(ieb) \\
		& = & -i(\mathbf{1}-Te)TeTa +(\mathbf{1}-Te)(-iT(ie))(-iT(ib)) \\
		& = & 0 +(\mathbf{1}-Te)Te(-iT(ib)) \\
		& = & 0,
	\end{eqnarray*}
and so, $(\mathbf{1}-Te)T(ex) = 0$.
Similarly, $TeT((\mathbf{1}-e)x) = 0$. From this we now see that $T(ex) = TeT(ex) = TeTx$ as claimed.
\end{proof}

\section{The von Neumann Algebra Case}\label{sect:von-neumann}
\noindent
We are now in a position to prove our main result.

\begin{thrm}\label{youbeauty}
Suppose that $T\colon A\to B$ is a surjective additive spectrum preserving mapping from a von Neumann algebra $A$ onto a unital Banach algebra~$B$.
Then $T$ is a Jordan isomorphism.
\end{thrm}
\begin{proof}
It suffices to show that $T = S$, cf.\ Proposition~\ref{genjordanprop}.
By Lemma~\ref{lem4}, $\tilde{R} + \alpha W$ is a Jordan automorphism on $A$ for all $\alpha$ sufficiently close to~$1$.
Fix such $\alpha\neq1$, and note that $H= \tilde{R} + \alpha W$ is in the principal component of the group of Jordan automorphisms on~$A$.
By Sinclair's result and its proof (see \cite[Lemma 1 and Theorem 2]{Sinclair1970}), $H = e^D$, where $D\colon A\to A$ is a derivation.
Since every derivation on a von Neumann algebra is inner (e.g., \cite[Corollary 8.6.6]{Ped1979}),
it follows that there is some $y \in A$ such that $Dx = yx-xy = [y, x]$ for all $x\in A$.
Consequently, $Hx = (e^D)(x) = e^{y}xe^{-y}$, that is, $H$ is an inner automorphism. Set $u=e^y$.
Since $\id - H = (1-\alpha)W$, it follows from Lemma~\ref{lem3} that $\id - H$ maps into the square zero elements in~$A$ and therefore,
$r(x - uxu^{-1}) = 0$ for all $x\in A$.
Replacing $x$ by $xu$, the aforementioned identity becomes $r(xu - ux) = 0$ for all $x \in A$.
Using Le Page's result \cite[Theorem 5.2.1]{aupetit1991primer}, we now infer that $u$ must be central in~$A$.
Consequently, $H =\id$, and so, $W = 0$ (since $\alpha\neq1$) and therefore $iR(x) = R(ix)$ for every $x\in A$.
Hence, $R$ is linear. Since $Rx = x$ for all $x \in A_{sa}$, $R = \id$ and $S = T$ which was to prove.
\end{proof}

\begin{cor}\label{cor:gen-case}
Every surjective additive spectrum preserving mapping from a \C* with real rank zero onto a unital Banach algebra is a Jordan isomorphism.
\end{cor}
\begin{proof}
Let $T$ be such a mapping. Continuing to use the above notation and results, we have $H=e^D$ for some derivation $D$ on~$A$.
We can extend $D$ and $H$ to the enveloping von Neumann algebra $A''$ to obtain a derivation $D''$, which is inner, and an automorphism $H''$
such that $H''=e^{D''}$. Consequently, $H''x=uxu^{-1}$ for an invertible element $u\in A''$ and all $x\in A''$.
Since $A$ is strongly dense in $A''$, it follows that $\id - H''$ maps into the square zero elements in~$A''$.
Thus, for each $x\in A''$, $r(x - uxu^{-1}) = 0$ and the argument is completed as in the second part of the proof of Theorem~\ref{youbeauty}.
\end{proof}

\medskip\noindent
\textbf{Acknowledgement.} The second author was supported by a grant from the National Research Foundation of South Africa (NRF Grant Number: 129692).

\end{document}